\documentclass[11pt,twoside]{article}
\usepackage{a4}
\usepackage{amssymb,amsmath,amsthm,latexsym}
\usepackage{amsfonts}
  \usepackage{amsfonts}
\usepackage{graphicx}

\newtheorem{theorem}{Theorem}[section]

\newtheorem{definition}[theorem]{Definition}

\newtheorem{lemma} [theorem]{Lemma}

\vspace{5cm}

\begin{document}
  
  \label{'ubf'}  
\setcounter{page}{1}                                 

\markboth {\hspace*{-9mm} \centerline{\footnotesize \sc
  A combinatorial  approach to the stronger Central Sets Theorem }
                 }
                { \centerline                           {\footnotesize \sc  
         Pintu Debnath                                                 } \hspace*{-9mm}              
               }

\vspace*{-2cm}

\begin{center}
{ 
       {\Large \textbf { \sc A combinatorial  approach to the stronger Central Sets Theorem for semigroups
                               }
       }
\\

\medskip

\author{Debnath, Pintu}
{\sc Pintu Debnath }\\
{\footnotesize Department of Mathematics,
			Basirhat College,
			Basirhat-743412, North 24th Parganas, West Bengal, India.}\\

{\footnotesize e-mail: {\it pintumath1989@gmail.com}}
}
\end{center}

\thispagestyle{empty}

\hrulefill



\begin{abstract}
H. Furstenberg introduced the notion  of central sets in terms
of topological dynamics and established the famous Central Sets Theorem. Later in [A new and stronger Central Sets Theorem, Fund. Math. 199 (2008), 155-175], D. De, N. Hindman, and D. Strauss established a stronger version of the  Central Sets Theorem that uses the algebra of the	Stone-\v Cech compactification of discrete semigroups. In this article, We  will provide a new and  combinatorial proof  of  the stronger Central Sets Theorem.
\end{abstract}

	

\section{Introduction}

In $1981$, H. Furstenberg \cite{F81} introduced the notions of Central sets using topological dynamics and proved the joint extension of two famous theorems:  one is the classical and possibly one of the first result of Ramsey theory; say \textit{van der Waerden's theorem} (\cite{vdw}), and the second one is the one of the first infinitary result in Ramsey theory; say \textit{Hindman theorem} (\cite{H74}). A few years later, in \cite{BH90}, V. Bergelson and N. Hindman established an equivalent definition of the Central sets. Basically, if $(S,\cdot )$ is a discrete semigroup, and $(\beta S,\cdot )$ is the corresponding semigroup of ultrafilters, then the Central sets are the members of the minimal idempotent ultrafilters. Later, this foundation of the notions of Central sets played a major role in the development of the \textit{Arithmetic Ramsey Theory:} most of the Ramsey theoretic configurations later found in the Central sets. After some seminal papers (for details see \cite{H20}), finally in \cite{DHS08}, D. De, N. Hindman, and D. Strauss proved a stronger version of this Central Sets Theorem (SCST). For some other versions of {\it SCST}, one can see \cite{P15, J15}.

For a general semigroup $(S,\cdot)$, a subset $A\subseteq S$ is said to be
\emph{syndetic} in $(S,\cdot)$ if there exists a finite nonempty set
$F\subseteq S$ such that
\[
\bigcup_{t\in F} t^{-1}A = S,
\]
where
\[
t^{-1}A = \{\, s\in S : t\cdot s \in A \,\}.
\]

A subset $A\subseteq S$ is said to be \emph{thick} if for every finite
nonempty subset $E\subseteq S$, there exists an element $x\in S$ such that
\[
E\cdot x \subseteq A.
\]

A subset $A\subseteq S$ is called \emph{piecewise syndetic} if there exists
a finite nonempty set $F\subseteq S$ such that
\[
\bigcup_{t\in F} t^{-1}A
\]
is thick in $S$. It is well known that a piecewise syndetic set can be
expressed as the intersection of a thick set and a syndetic set.

\begin{definition}
Let $(S,\cdot)$ be a semigroup and let $\mathcal{A}\subseteq \mathcal{P}(S)$.
The family $\mathcal{A}$ is called \textbf{collectionwise piecewise syndetic}
if and only if there exist functions
\[
G:\mathcal{P}_{f}(\mathcal{A})\to \mathcal{P}_{f}(S)
\quad \text{and} \quad
x:\mathcal{P}_{f}(\mathcal{A})\times \mathcal{P}_{f}(S)\to S
\]
such that for all $F\in \mathcal{P}_{f}(S)$ and all
$\mathcal{F},\mathcal{H}\in \mathcal{P}_{f}(\mathcal{A})$ with
$\mathcal{F}\subseteq \mathcal{H}$, one has
\[
F\cdot x(\mathcal{H},\mathcal{F})
\subseteq
\bigcup_{t\in G(\mathcal{F})} t^{-1}\!\left(\bigcap \mathcal{F}\right).
\]
\end{definition}

The following classical Central Sets Theorem is due to H.~Furstenberg \cite{F81}.

\begin{theorem}
Let $A$ be a central subset of $\mathbb{N}$, let $k\in\mathbb{N}$, and for
each $i\in\{1,2,\dots,k\}$ let
$\langle y_{i,n}\rangle_{n=1}^{\infty}$ be a sequence in $\mathbb{Z}$.
Then there exist sequences
$\langle a_n\rangle_{n=1}^{\infty}$ in $\mathbb{N}$ and
$\langle H_n\rangle_{n=1}^{\infty}$ in $\mathcal{P}_f(\mathbb{N})$ such that
\begin{itemize}
\item[(1)] $\max H_n < \min H_{n+1}$ for each $n\in\mathbb{N}$; and
\item[(2)] for each $i\in\{1,2,\dots,k\}$ and each
$F\in\mathcal{P}_f(\mathbb{N})$,
\[
\sum_{n\in F}\left(a_n+\sum_{t\in H_n}y_{i,t}\right)\in A.
\]
\end{itemize}
\end{theorem}

\begin{definition}\label{non commu cha central}
Let $A$ be a subset of a semigroup $S$.
The set $A$ is called \textbf{central} in $S$ if there exists a downward
directed family $\langle C_F\rangle_{F\in I}$ of subsets of $A$ such that
\begin{itemize}
\item[(i)] for each $F\in I$ and each $x\in C_F$, there exists $G\in I$
with $C_G\subseteq x^{-1}C_F$; and
\item[(ii)] the family $\{C_F : F\in I\}$ is collectionwise piecewise syndetic.
\end{itemize}
\end{definition}

The following theorem gives the Central Sets Theorem for infinite
commutative semigroups \cite[Theorem~14.11]{HS12}.

\begin{theorem}\label{CST com}
Let $(S,+)$ be an infinite commutative semigroup and let $A$ be a central
subset of $S$. For each $l\in\mathbb{N}$, let
$\langle y_{l,n}\rangle_{n=1}^{\infty}$ be a sequence in $S$.
Then there exist sequences
$\langle a_n\rangle_{n=1}^{\infty}$ in $S$ and
$\langle H_n\rangle_{n=1}^{\infty}$ in $\mathcal{P}_f(\mathbb{N})$ such that
$\max H_n<\min H_{n+1}$ for all $n\in\mathbb{N}$ and
\[
FS\!\left(\left\langle a_n+\sum_{t\in H_n}y_{f(n),t}\right\rangle_{n=1}^{\infty}\right)
\subseteq A
\]
for each $f\in\Phi$, where $\Phi$ is the set of all functions
$f:\mathbb{N}\to\mathbb{N}$ satisfying $f(n)\leq n$ for all $n\in\mathbb{N}$.
\end{theorem}

For $m\in\mathbb{N}$, define
\[
\mathcal{H}_m=\left\{
(H_1,\dots,H_m)\in\bigl(\mathcal{P}_f(\mathbb{N})\bigr)^m :
\max H_t<\min H_{t+1}\text{ for }1\le t<m
\right\}.
\]

The following theorem is the most general form of the Central Sets Theorem
prior to the stronger version in \cite{DHS08}
(see \cite[Theorem~14.15]{HS12}).

\begin{theorem}\label{CST noncom}
Let $S$ be a semigroup and let $A$ be a central subset of $S$.
For each $l\in\mathbb{N}$, let
$\langle y_{l,n}\rangle_{n=1}^{\infty}$ be a sequence in $S$.
Given $l,m\in\mathbb{N}$, $a\in S^{m+1}$, and $H\in\mathcal{H}_m$, define
\[
\omega(a,H,l)
=
\left(\prod_{i=1}^{m}
\left(a(i)\cdot\prod_{t\in H(i)}y_{l,t}\right)\right)
\cdot a(m+1).
\]

There exist sequences
$\langle m(n)\rangle_{n=1}^{\infty}$,
$\langle a_n\rangle_{n=1}^{\infty}$, and
$\langle H_n\rangle_{n=1}^{\infty}$ such that
\begin{itemize}
\item[(1)] for each $n\in\mathbb{N}$,
$m(n)\in\mathbb{N}$, $a_n\in S^{m(n)+1}$,
$H_n\in\mathcal{H}_{m(n)}$, and
$\max H_{n,m(n)}<\min H_{n+1,1}$; and
\item[(2)] for each $f\in\Phi$,
\[
FP\!\left(
\left\langle
\omega\bigl(a_n,H_n,f(n)\bigr)
\right\rangle_{n=1}^{\infty}
\right)
\subseteq A.
\]
\end{itemize}
\end{theorem}

In \textbf{ Section 2 } and \textbf{ Section 3 } we will prove a stronger version of Theorem \ref{CST com} and \ref{CST noncom}  respectively by combinatorial the characterization of central sets.

\section{Commutative stronger central set theorem}

 In this section, we will prove the following from \cite[ Theorem 14.8.4]{HS12} for commutative semigroup.

\begin{theorem}\label{central commutative}

Let $\left(S,+\right)$ be a  commutative semigroup. Let $C$ be a
central subset of $S$. Then there exist functions $\alpha:\mathcal{P}_{f}\left(^{\mathbb{N}}S\right)\rightarrow S$
such that 
\begin{itemize}
    \item[(1)] let $F,G\in\mathcal{P}_{f}\left(^{\mathbb{N}}S\right)$
and $F\subsetneq G$, then $\max H\left(F\right)<\min H\left(G\right)$,

\item[(2)]  whenever $r\in\mathbb{N}$, $G_{1},G_{2},...,G_{r}\in\mathcal{P}_{f}\left(^{\mathbb{N}}S\right)$
such that $G_{1}\subsetneq G_{2}\subsetneq\ldots\subsetneq G_{r}$
and for each $i\in\left\{ 1,2,\ldots,r\right\} $, $f_{i}\in G_{i}$
one has 
\[
\sum_{i=1}^{r}\left(\alpha\left(G_{i}\right)+\sum_{t\in H\left(G_{i}\right)}f_{i}\left(t\right)\right)\in C.
\]
\end{itemize}

\end{theorem}

From \cite[Definition~14.8.1]{HS12}, we recall the notion of $J$-sets,
which play a crucial role in the proof of the stronger form of the
Central Sets Theorem.

\begin{definition}
Let $(S,+)$ be a commutative semigroup and let $A\subseteq S$.
The set $A$ is called a \textbf{$J$-set} if and only if, for every
$F\in\mathcal{P}_f({}^{\mathbb{N}}S)$, there exist
$a\in S$ and $H\in\mathcal{P}_f(\mathbb{N})$ such that
\[
a+\sum_{t\in H} f(t)\in A
\quad \text{for all } f\in F.
\]
\end{definition}

From \cite[Lemma~14.8.2]{HS12}, we obtain the following stronger property
of $J$-sets, which will be used in the proof of our main result in this
section.

\begin{lemma}\label{J set commit stron}
Let $(S,+)$ be a commutative semigroup and let $A\subseteq S$ be a
$J$-set.
Then, for every $m\in\mathbb{N}$ and every
$F\in\mathcal{P}_f({}^{\mathbb{N}}S)$, there exist
$a\in S$ and $H\in\mathcal{P}_f(\mathbb{N})$ such that
\[
\min H > m
\quad\text{and}\quad
a+\sum_{t\in H} f(t)\in A
\quad \text{for all } f\in F.
\]
\end{lemma}

We will use the Hales--Jewett Theorem \cite{HJ63} to show that every piecewise
syndetic set is a $J$-set. We begin with a brief review of the necessary
combinatorial terminology.

For $t\in\mathbb{N}$, let
\[
[t]=\{1,2,\ldots,t\}.
\]
Words of length $N$ over the alphabet $[t]$ are the elements of $[t]^N$.
A \emph{variable word} is a word over the alphabet
$[t]\cup\{*\}$ in which the symbol $*$ occurs at least once and serves as
a variable.
Given a variable word $\tau(*)$, the associated
\emph{combinatorial line} is
\[
L_{\tau}
=
\bigl\{
\tau(1),\tau(2),\ldots,\tau(t)
\bigr\},
\]
where $\tau(i)$ is obtained by replacing every occurrence of $*$ in
$\tau(*)$ by $i$.

\begin{theorem}[Hales--Jewett]\label{HJ}
For all $r,t\in\mathbb{N}$, there exists a number $HJ(r,t)$ such that,
whenever $N\geq HJ(r,t)$ and the set $[t]^N$ is colored with $r$ colors,
there exists a monochromatic combinatorial line.
\end{theorem}

\begin{lemma}\label{PiceJcomu}
Every piecewise syndetic set is a $J$-set.
\end{lemma}

\begin{proof}
Let $(S,+)$ be a commutative semigroup, and let $A\subseteq S$ be a piecewise syndetic set. 
Then there exists a finite set $E\subseteq S$ such that
\[
\bigcup_{t\in E}(-t+A)
\]
is thick. Write $|E|=r$.

Let $F\in \mathcal{P}_f({}^{\mathbb{N}}S)$ be arbitrary, and write $|F|=n$.
Enumerate $F$ as
\[
F=\{f_1,f_2,\dots,f_n\}.
\]
Let $N=N(r,n)$ be the Hales--Jewett number, and set $G=[n]^N$.

Define a map $g\colon G\to S$ by
\[
g(a_1,a_2,\dots,a_N)=\sum_{i=1}^{N} f_{a_i}(i).
\]
Since $g(G)$ is finite and $\bigcup_{t\in E}(-t+A)$ is thick, there exists $b\in S$ such that
\[
b+g(G)\subseteq \bigcup_{t\in E}(-t+A).
\]

This induces an $r$-coloring $\chi$ of $G$ defined by
\[
\chi(a)=i \quad \text{if and only if} \quad b+g(a)\in -t_i + A,
\]
where $t_1,\dots,t_r$ enumerate the elements of $E$, and $i$ is chosen to be the least such index.

By the Hales--Jewett Theorem, there exists a monochromatic combinatorial line in $G$.
This yields a finite nonempty set $H\subseteq \mathbb{N}$ and an element $a\in S$ such that
\[
b+a+\sum_{t\in H} f_i(t)\in A \quad \text{for all } f_i\in F.
\]
Equivalently, there exists $s\in S$ such that
\[
s+\sum_{t\in H} f(t)\in A \quad \text{for all } f\in F.
\]
Hence $A$ is a $J$-set.
\end{proof}

Now we are in the position to prove the main result of this section.

\begin{proof}[\textbf{Proof of Theorem~\ref{central commutative}}]

Let $A$ be a central set in the commutative semigroup $(S,+)$.
By Definition~\ref{non commu cha central}, there exists a downward directed family
$\langle A_N\rangle_{N\in I}$ of subsets of $A$ such that
$\{A_N : N\in I\}$ is collectionwise piecewise syndetic and has the following property:
for each $N\in I$ and each $x\in A_N$, there exists $M\in I$ with
\[
A_M \subseteq -x + A_N.
\]

Fix $N\in I$ and consider the piecewise syndetic set $A_N$.

We define functions
\[
\alpha(F)\in S \quad \text{and} \quad H(F)\in \mathcal{P}_f(\mathbb{N}),
\]
for each $F\in \mathcal{P}_f({}^{\mathbb{N}}S)$, by induction on $|F|$,
so that the following conditions are satisfied:

\begin{itemize}
\item[(1)]
If $F,G\in \mathcal{P}_f({}^{\mathbb{N}}S)$ and
$\emptyset\neq G\subsetneq F$, then
\[
\max H(G) < \min H(F).
\]

\item[(2)]
Whenever $n\in\mathbb{N}$,
$G_1\subsetneq G_2\subsetneq\cdots\subsetneq G_n=F$,
and $(f_i)_{i=1}^n\in \prod_{i=1}^n G_i$, we have
\[
\sum_{i=1}^{n}\left(
\alpha(G_i)+\sum_{t\in H(G_i)} f_i(t)
\right)\in A_N.
\]
\end{itemize}

\medskip

\noindent
\textbf{Base step.}
Let $F=\{f\}$.
By Lemma~\ref{PiceJcomu}, the set $A_N$ is a $J$-set.
Hence there exist $a\in S$ and $L\in \mathcal{P}_f(\mathbb{N})$ such that
\[
a+\sum_{t\in L} f(t)\in A_N.
\]
Define $\alpha(\{f\})=a$ and $H(\{f\})=L$.

\medskip

\noindent
\textbf{Inductive step.}
Assume $|F|>1$ and that $\alpha(G)$ and $H(G)$ have been defined
for all nonempty proper subsets $G$ of $F$ satisfying the inductive hypotheses.

Let
\[
K=\bigcup\{H(G): \emptyset\neq G\subsetneq F\},
\qquad
m=\max K.
\]

Define
\[
\begin{aligned}
C=\Bigl\{
&\sum_{i=1}^{n}\Bigl(
\alpha(G_i)+\sum_{t\in H(G_i)} f_i(t)
\Bigr):
n\in\mathbb{N}, \\
&\quad G_1\subsetneq G_2\subsetneq\cdots\subsetneq G_n\subsetneq F,\;
(f_i)_{i=1}^n\in \prod_{i=1}^n G_i
\Bigr\}.
\end{aligned}
\]

The set $C$ is finite and $C\subseteq A_N$ by the inductive hypothesis.
Let
\[
B = A_N \cap \bigcap_{c\in C} (-c + A_N).
\]

For each $c\in C$, there exists $N(c)\in I$ such that
$A_{N(c)}\subseteq -c + A_N$.
Since the family $\langle A_N\rangle_{N\in I}$ is downward directed,
there exists $M\in I$ such that
\[
A_M \subseteq B.
\]
Thus $B$ is a $J$-set.

By Lemma~\ref{J set commit stron}, there exist
$a\in S$ and $L\in \mathcal{P}_f(\mathbb{N})$ such that
\[
\min L > m
\quad\text{and}\quad
a+\sum_{t\in L} f(t)\in A_M
\quad \text{for all } f\in F.
\]

Define $\alpha(F)=a$ and $H(F)=L$.
Then condition~(1) is satisfied by the choice of $L$.
Moreover, condition~(2) follows since
\[
\sum_{i=1}^{n}\left(
\alpha(G_i)+\sum_{t\in H(G_i)} f_i(t)
\right)
+
\left(
a+\sum_{t\in L} f(t)
\right)
\in A_N,
\]
for all chains $G_1\subsetneq\cdots\subsetneq G_n=F$.

This completes the inductive construction and hence the proof
of the Central Sets Theorem.
\end{proof}

\section{Noncommutative stronger Central set theorem}

For $m\in\mathbb{N}$, define
\[
\mathcal{J}_m
=
\Bigl\{
\bigl(t(1),t(2),\ldots,t(m)\bigr)\in \mathbb{N}^m :
t(1)<t(2)<\cdots<t(m)
\Bigr\}.
\]

In this section, we present a combinatorial proof of the following result,
which is a stronger version of Theorem~\ref{CST noncom}.

\begin{theorem}\label{noncom central}
Let $(S,\cdot)$ be a semigroup and let $A\subseteq S$ be a central set.
Then there exist
\[
m \colon \mathcal{P}_f({}^{\mathbb{N}}S)\to\mathbb{N},
\qquad
\alpha \in \prod_{F\in\mathcal{P}_f({}^{\mathbb{N}}S)} S^{\,m(F)+1},
\]
and
\[
\tau \in \prod_{F\in\mathcal{P}_f({}^{\mathbb{N}}S)} \mathcal{J}_{m(F)},
\]
such that the following conditions hold:
\begin{enumerate}
\item
If $F,G\in \mathcal{P}_f({}^{\mathbb{N}}S)$ and $F\subset G$, then
\[
\tau(F)\bigl(m(F)\bigr) < \tau(G)(1).
\]

\item
Whenever $n\in\mathbb{N}$,
$G_1\subset G_2\subset \cdots \subset G_n$ are elements of
$\mathcal{P}_f({}^{\mathbb{N}}S)$, and for each
$i\in\{1,2,\ldots,n\}$ we choose $f_i\in G_i$, one has
\[
\prod_{i=1}^{n}
x\bigl(m(G_i),\alpha(G_i),\tau(G_i),f_i\bigr)
\in A.
\]
\end{enumerate}
\end{theorem}

 Now, we are inviting $J$-sets from \cite[Definition 14.14.1]{HS12} in noncommutative settings.

\begin{definition}
Let $(S,\cdot)$ be a semigroup.
\begin{enumerate}
\item
Let
\[
\mathcal{T} = {}^{\mathbb{N}}S.
\]

\item
Given $m\in\mathbb{N}$, $a\in S^{m+1}$, $t\in\mathcal{J}_m$, and
$f\in\mathcal{T}$, define
\[
x(m,a,t,f)
=
\left(\prod_{j=1}^{m} \bigl(a(j)\cdot f(t(j))\bigr)\right)\cdot a(m+1).
\]

\item
A subset $A\subseteq S$ is called a \emph{$J$-set} if and only if,
for every $F\in\mathcal{P}_f(\mathcal{T})$, there exist
$m\in\mathbb{N}$, $a\in S^{m+1}$, and $t\in\mathcal{J}_m$
such that
\[
x(m,a,t,f)\in A \quad \text{for all } f\in F.
\]
\end{enumerate}
\end{definition}

We record the following consequence of \cite[Lemma~14.14.3]{HS12}.

\begin{lemma}\label{J set noncom stro}
Let $S$ be a semigroup and let $A\subseteq S$ be a $J$-set.
Then for each $F\in\mathcal{P}_f({}^{\mathbb{N}}S)$ and each $n\in\mathbb{N}$,
there exist $m\in\mathbb{N}$, $a\in S^{m+1}$, and $t\in\mathcal{J}_m$
such that $t(1)>n$ and
\[
x(m,a,t,f)\in A \quad \text{for all } f\in F.
\]
\end{lemma}

To prove Theorem~\ref{noncom central}, we first show that
piecewise syndetic sets are $J$-sets in the noncommutative setting.

\begin{theorem}
Let $(S,\cdot)$ be a semigroup and let $A\subseteq S$ be a piecewise syndetic set.
Then $A$ is a $J$-set.
\end{theorem}

\begin{proof}
Let $A\subseteq S$ be piecewise syndetic.
Then there exists a finite set $F\subseteq S$ such that
\[
\bigcup_{x\in F} x^{-1}A
\]
is thick.

Let $G\in \mathcal{P}_f({}^{\mathbb{N}}S)$ and write
\[
G=\{f_1,f_2,\ldots,f_k\}.
\]
Let $|F|=r$, and let $N=N(k,r)$ be the Hales--Jewett number
guaranteed by Theorem~\ref{HJ}.

Define a correspondence map from $[k]^N$ into $S$ by
\[
(i_1,i_2,\ldots,i_N)
\longmapsto
f_{i_1}(1)\cdot f_{i_2}(2)\cdots f_{i_N}(N).
\]
Let
\[
H=
\Bigl\{
f_{i_1}(1)\cdot f_{i_2}(2)\cdots f_{i_N}(N)
:
(i_1,i_2,\ldots,i_N)\in [k]^N
\Bigr\}.
\]
Then $H$ is finite.

Since $\bigcup_{x\in F} x^{-1}A$ is thick, there exists $s\in S$ such that
\[
H\cdot s \subseteq \bigcup_{x\in F} x^{-1}A.
\]
Thus the set $H\cdot s$ admits an $r$-coloring.

Define a coloring $\chi$ of $[k]^N$ by
\[
\chi(i_1,i_2,\ldots,i_N)
=
\chi\bigl(
f_{i_1}(1)\cdot f_{i_2}(2)\cdots f_{i_N}(N)\cdot s
\bigr).
\]

By the Hales--Jewett Theorem, there exists a monochromatic
combinatorial line in $[k]^N$.
This yields elements $a_1,a_2,\ldots,a_{n+1}\in S$ and
$t=(t_1,t_2,\ldots,t_n)\in\mathcal{J}_n$ such that
\[
\Bigl\{
a_1\cdot f(t_1)\cdot a_2\cdot f(t_2)\cdots
a_n\cdot f(t_n)\cdot a_{n+1}
:
f\in G
\Bigr\}
\subseteq x^{-1}A
\]
for some $x\in F$.

Multiplying on the left by $x$, we obtain
\[
\Bigl\{
x\cdot a_1\cdot f(t_1)\cdot a_2\cdot f(t_2)\cdots
a_n\cdot f(t_n)\cdot a_{n+1}
:
f\in G
\Bigr\}
\subseteq A.
\]

Define
\[
a(1)=x\cdot a_1,
\qquad
a(i)=a_i \ \text{ for } i=2,\ldots,n+1,
\qquad
t=(t_1,t_2,\ldots,t_n).
\]
Then
\[
x(n,a,t,f)\in A \quad \text{for all } f\in G,
\]
and hence $A$ is a $J$-set.
\end{proof}

It is the right time to prove the main result of this section.

\begin{proof}[\textbf{Proof of Theorem~\ref{noncom central}}]

Let $A$ be a central set in the semigroup $(S,\cdot)$.
By Definition~\ref{non commu cha central}, there exists a downward directed family
$\langle A_N\rangle_{N\in I}$ of subsets of $A$ such that
$\{A_N : N\in I\}$ is collectionwise piecewise syndetic and has the property that
for each $N\in I$ and each $x\in A_N$, there exists $M\in I$ with
\[
A_M \subseteq x^{-1}A_N.
\]

Fix $N\in I$ and consider the piecewise syndetic set $A_N$.

We define $m(F)\in\mathbb{N}$, $\alpha(F)\in S^{m(F)+1}$, and
$\tau(F)\in\mathcal{J}_{m(F)}$ for each
$F\in\mathcal{P}_f({}^{\mathbb{N}}S)$ by induction on $|F|$,
so that the following conditions hold:
\begin{enumerate}
\item
If $\emptyset\neq G\subsetneq F$, then
\[
\tau(G)\bigl(m(G)\bigr) < \tau(F)(1).
\]

\item
Whenever $n\in\mathbb{N}$,
$\emptyset\neq G_1\subsetneq G_2\subsetneq\cdots\subsetneq G_n=F$,
and $f_i\in G_i$ for each $i\in\{1,2,\ldots,n\}$, one has
\[
\prod_{i=1}^{n}
x\bigl(m(G_i),\alpha(G_i),\tau(G_i),f_i\bigr)
\in A_N.
\]
\end{enumerate}

\medskip

\noindent
\textbf{Base step.}
Let $F=\{f\}$.
Since $A_N$ is a $J$-set, choose
$m(F)\in\mathbb{N}$, $\alpha(F)\in S^{m(F)+1}$, and
$\tau(F)\in\mathcal{J}_{m(F)}$ such that
\[
x\bigl(m(F),\alpha(F),\tau(F),f\bigr)\in A_N.
\]

\medskip

\noindent
\textbf{Inductive step.}
Assume $|F|>1$ and that $m(G)$, $\alpha(G)$, and $\tau(G)$
have been defined for all nonempty proper subsets $G$ of $F$.
Let
\[
k=\max\bigl\{\tau(G)\bigl(m(G)\bigr):\emptyset\neq G\subsetneq F\bigr\}.
\]

Define
\[
\begin{aligned}
C=\Bigl\{
&\prod_{i=1}^{n}
x\bigl(m(G_i),\alpha(G_i),\tau(G_i),f_i\bigr):
n\in\mathbb{N},\\
&\emptyset\neq G_1\subsetneq G_2\subsetneq\cdots\subsetneq G_n\subsetneq F,
\; f_i\in G_i
\Bigr\}.
\end{aligned}
\]

By the inductive hypothesis, $C$ is a finite subset of $A_N$.
Let
\[
B = A_N \cap \bigcap_{c\in C} c^{-1}A_N.
\]

For each $c\in C$, there exists $N(c)\in I$ such that
$A_{N(c)}\subseteq c^{-1}A_N$.
Since the family $\langle A_N\rangle_{N\in I}$ is downward directed,
there exists $M\in I$ such that
\[
A_M \subseteq B.
\]
Hence $B$ is a $J$-set.

Applying Lemma~\ref{J set noncom stro}, choose
$m(F)\in\mathbb{N}$, $\alpha(F)\in S^{m(F)+1}$, and
$\tau(F)\in\mathcal{J}_{m(F)}$ such that
\[
\tau(F)(1)>k
\quad\text{and}\quad
x\bigl(m(F),\alpha(F),\tau(F),f\bigr)\in A_M
\quad\text{for all } f\in F.
\]

Condition~(1) is immediate from the choice of $\tau(F)$.
To verify condition~(2), let $n\in\mathbb{N}$ and
\[
\emptyset\neq G_1\subsetneq G_2\subsetneq\cdots\subsetneq G_n=F,
\quad f_i\in G_i.
\]

If $n=1$, then
\[
x\bigl(m(G_1),\alpha(G_1),\tau(G_1),f_1\bigr)\in A_M.
\]
Assume $n>1$ and let
\[
b=\prod_{i=1}^{n-1}
x\bigl(m(G_i),\alpha(G_i),\tau(G_i),f_i\bigr).
\]
Then $b\in A_M$, and hence
\[
x\bigl(m(G_n),\alpha(G_n),\tau(G_n),f_n\bigr)\in b^{-1}A_M,
\]
which implies
\[
\prod_{i=1}^{n}
x\bigl(m(G_i),\alpha(G_i),\tau(G_i),f_i\bigr)\in A_M\subseteq A_N.
\]

This completes the induction and the proof.
\end{proof}

\bibliographystyle{plain}

\begin{thebibliography}{9}



\bibitem{BH90} V. Bergelson and N. Hindman,  Nonmetrizable topological dynamics and Ramsey theory. Transactions of the American Mathematical Society, 320 (1990), 293–320. https://doi.org/10.1090/s0002-9947-1990-0982232-5.

\bibitem{DHS08} D. De, N. Hindman, and D. Strauss, A new and stronger Central Sets Theorem, Fund. Math. 199 (2008), 155-175.  https://doi.org/10.4064/fm199-2-5. 


\bibitem{F81} H. Furstenberg: Recurrence in Ergodic Theory and Combinatorial Number Theory, Princeton University Press, 1981.

\bibitem{HJ63} A. Hales and R. Jewett, Regularity and positional games, Trans. Amer. Math. Soc. 106 (1963) 222-229. https://doi.org/10.1007/978-0-8176-4842-8-23.

\bibitem{H74} N. Hindman: Finite sums from sequences within cells of partitions of $\mathbb{N}$, 
J. Combin. Theory Ser. A \textbf{17} (1974), 1–11.

\bibitem{H20} N. Hindman, A history of central sets, Ergodic Theory Dynam. Systems 40(2020) no 1, 1-33.  https://doi.org/10.1017/etds.2018.37.


\bibitem{HS12} N. Hindman, and D. Strauss, Algebra in the	Stone-\v Cech compactification: theory and applications, second edition, de Gruyter, Berlin, 2012.

\bibitem{P15} D. Phulara, A generalized central sets theorem and applications. Topology and Its Applications, 196 (2015), 92–105. https://doi.org/10.1016/j.topol.2015.09.038.	


\bibitem{J15} J. H. Johnson Jr., A new and simpler noncommutative central sets theorem. Topology and Its Applications, 189 (2015) 10–24. https://doi.org/10.1016/j.topol.2015.03.006.



\bibitem{vdw} B.L. van der Waerden. Beweis einer baudetschen vermutung. \emph{Nieuw. Arch. Wisk.}, 15: 212--216, (1927).

\end{thebibliography}

\end{document}